\documentclass[a4paper,11pt, reqno]{amsart}

\usepackage{amsmath, amssymb, amsthm, xcolor, enumerate, array, graphicx, bbm}
\usepackage[all]{xy}
\usepackage{hyperref}
\usepackage{aliascnt}
\usepackage{todonotes}

\setuptodonotes{fancyline, color=orange!15, bordercolor=orange}
\newcounter{todocounter}

\newtheorem{theorem}{Theorem}[section]
\newtheorem{lemma}[theorem]{Lemma}

\theoremstyle{definition}
\newtheorem{definition}[theorem]{Definition}

\theoremstyle{remark}
\newtheorem*{remark*}{Remark}
\newtheorem{remark}[theorem]{Remark}

\numberwithin{equation}{section}
\allowdisplaybreaks

\newcommand{\K}{\mathbb{K}}

\newcommand{\Z}{\mathcal{Z}}
\newcommand{\Zo}{{\mathcal{Z}_0}}

\newcommand{\C}{\ensuremath{\mathrm{C}^*}}

\newcommand{\id}{\mathrm{id}}
\newcommand{\tr}{\mathrm{tr}}

\newcommand{\W}{\mathcal{W}}

\newcommand{\bsub}[1]{_{\scalebox{1.1}{$\scriptstyle #1$}}}

\title[A homotopy rigidity theorem for $\Zo$-stable $\C$-algebras]{A homotopy rigidity theorem for $\Zo$-stable $\C$-algebras}

\author[J.\ Castillejos]{Jorge Castillejos}
\address{\hskip-\parindent Jorge Castillejos, Unidad Cuernavaca del Instituto de Matematicas, UNAM, Av. Universidad s/n, 62210 Cuernavaca, Morelos, México}
\email{jorge.castillejos@im.unam.mx}

\author[B.\ Debets]{Baukje Debets} 
\email{}

\author[G.\ Szab\'o]{G\'abor Szab\'o}
\address{\hskip-\parindent G\'abor Szab\'o, Department of Mathematics, KU Leuven, Celestijnenlaan 200B, box 2400, 3001 Leuven, Belgium}
\email{gabor.szabo@kuleuven.be}
\begin{document}
	
\maketitle

\begin{abstract}
We show that two simple, separable, nuclear and $\Zo$-stable $\C$-algebras are isomorphic if they are trace-preservingly homotopy equivalent.
This result does not assume the UCT and can be viewed as a tracial stably projectionless analog of the homotopy rigidity theorem for Kirchberg algebras.  
\end{abstract}

\renewcommand*{\thetheorem}{\Alph{theorem}}

\section*{Introduction}

A milestone in the classification programme of $\C$-algebras states that the class of simple separable nuclear unital and $\Z$-stable $\C$-algebras that satisfy the \emph{Universal Coefficient Theorem} \cite{RS1987} (UCT) are classified by an invariant constructed with $K$-theory and tracial data \cite{Gong2020, Gong2020a, Elliott2024, Tikuisis2017, Castillejos2021, CGSTW23}.
Therefore, determining when a $\C$-algebra satisfies these conditions is essential before one can attempt to apply the classification theorem.
Many criteria exist to detect $\Z$-stability in various contexts, and it is also known that there are many separable simple unital and nuclear \C-algebras that are not $\Z$-stable. 
However, a major open  problem in the field is determining if nuclear $\C$-algebras satisfy the UCT. We refer to \cite{RS1987, Skandalis1988} for a comprehensive description of the UCT.

The work of Tu shows that groupoid $\C$-algebras of \emph{a-T-menable} groupoids (in particular amenable ones) satisfy the UCT (\cite{Tu1999}), and this encompasses large classes of separable nuclear $\C$-algebras.
Nevertheless, all concrete known examples of nuclear $\C$-algebras are seen to satisfy the UCT and there is no apparent candidate for a counterexample (although such examples do exist outside the nuclear setting; see \cite{Skandalis1988}). 

In light of the difficulty of settling the UCT problem, it is natural to seek classification results where the UCT is not needed.
For instance, the classification theorem of Kirchberg and Phillips \cite{KirchbergUnp, Phillips2000, KP2000} for simple nuclear purely infinite \C-algebras has as a major methodological advantage that the classification is initially obtained directly via KK-theory.
The UCT only plays a role when one wants to obtain KK-equivalence from an isomorphism of ordinary K-theory.
In the classification of stably finite \C-algebras, all the available theories utilise the UCT assumption in several substantial intermediate steps.
This could be relaxed recently in a breakthrough article of Schafhauser \cite{Schafhauser2024}, in which an isomorphism theorem was proved for unital simple nuclear $\Z$-stable \C-algebras under the assumption that one starts from an embedding that induces both a $KK$-equivalence and an isomorphism of tracial data. 

In this note, we will focus on classification results without assuming the UCT for simple and stably projectionless nuclear $\C$-algebras.
In this direction, an important result for this note is the classification of KK-contractible (i.e.\ KK-equivalent to $\{0\}$) stably projectionless simple separable and $\Z$-stable $\C$-algebras via their tracial cones and scales \cite[Theorem 7.5]{Elliott2020}.
In particular, this applies to simple separable stably projectionless nuclear $\C$-algebras with traces that absorb tensorially the Razak--Jacelon algebra $\W$ \cite{Razak2002, Jacelon2013}.
Utilizing this result as a cornerstone, we establish a $\Zo$-stable uniqueness theorem for $^*$-homomorphisms between simple separable nuclear $\C$-algebras that are \emph{trace-preservingly homotopic} (see Definition \ref{def:tracial.homotopy} and Theorem \ref{maintheorem}).
Here $\Zo$ is the known stably projectionless analog of the Jiang--Su algebra that plays an important role in the classification of a large class of stably projectionless \C-algebras \cite{GL20}.
For classifiable \C-algebras with the UCT, the assumption of $\Zo$-stability is reflected in the Elliott invariant via the assumption that the pairing map between the traces and the $K_0$-group has to vanish.

By combining our aforementioned uniqueness theorem for maps with an Elliott intertwining argument, we obtain the following rigidity property for the class of separable, simple, nuclear and $\Zo$-stable $\C$-algebras.

\begin{theorem}
	Let $A$ and $B$ be simple, separable, nuclear and $\Zo$-stable $\C$-algebras. If $A$ and $B$ are trace-preservingly homotopy equivalent, then $A$ is isomorphic to $B$.
\end{theorem}

\subsection*{Acknowledgement.} 
JC was supported by UNAM--PAPIIT IA103124. 
BD was partially supported by European Research Council Consolidator Grant 614195--RIGIDITY.
GS was supported by research project G085020N funded by the Research Foundation Flanders (FWO), and the European Research Council under the European Union's Horizon Europe research and innovation programme (ERC grant AMEN--101124789).
Both BD and GS were partially supported by KU Leuven internal funds projects STG/18/019 and C14/19/088.

\numberwithin{theorem}{section}

\section{Preliminaries}

\subsection{Notation}

We will denote the multiplier algebra of $A$ by $\mathcal{M}(A)$ and its forced unitisation by $A^\dagger$.
If $A$ is unital, the unitary group is denoted by $\mathcal{U}(A)$.
We will write $\mathcal{U}(1+A)$ for the unitary subgroup $(1+A)\cap \mathcal{U}(A^\dagger)$.
Observe that if $A$ is unital, we can canonically identify $\mathcal{U}(A)$ with $\mathcal{U}(1+A)$. 

We will denote the $n\times n$-matrices with complex coefficients by $M_n(\mathbb{C})$.
We denote the standard matrix units by $(e_{ij})_{i,j=1}^n$.
We will also freely identify $M_n(A)$ with $M_n(\mathbb{C}) \otimes A$ whenever it is convenient for us. 
The cone of lower semicontinuous densely defined traces on $A$ is denoted by $T^+(A)$.

We will frequently  write $a \approx_\varepsilon b$ as short-hand for $\|a -b \| \leq \varepsilon$.
For $^*$-homomorphisms $\varphi, \psi: A \to B$, we write $\varphi \approx_{\mathrm{u}} \psi$ to say that they are \emph{approximately unitarily equivalent}, i.e., there is a net of unitaries $(u_\lambda) \subset \mathcal{U}(1+A)$ with $\lim\limits_{\lambda \to \infty} u_\lambda \varphi (a) u_\lambda^* = \psi(a)$ for all $a \in A$.
If one assumes $A$ to be separable, then such nets can be replaced by sequences.

Lastly, we shall say that a separable \C-algebra $A$ is \emph{$KK$-contractible} if $KK(A,A)=0$.

\subsection{Robert's classification theorem}

Given two positive elements $a$ and $b$ in a $\C$-algebra $A$, it is said that $a$ is \emph{Cuntz-below} $b$, $a \precsim b$, if for any $\varepsilon >0$ there is $x\in A$ such that $a \approx_\varepsilon x^*bx$. 
It is said that $a$ is \emph{Cuntz equivalent} to $b$, $a \sim b$, if $a \precsim b$ and $b \precsim a$.
The \emph{Cuntz semigroup} of $A$ is defined as $\mathrm{Cu}(A):= (A\otimes \K)_+ /_\sim$ equipped with orthogonal addition and order given by Cuntz-subequivalence. 
The equivalence class of a given element $a \in (A \otimes \K)_+$ will be denoted by $[a]$.

The \emph{augmented Cuntz semigroup} of a unital $\C$-algebra $A$, denoted by $\mathrm{Cu}^\sim (A)$, is defined  as the ordered semigroup of formal differences $[a] - n [1_A]$, with $a \in \mathrm{Cu}(A)$ and $n \in \mathbb{N}$, i.e.,
\begin{align}
	\mathrm{Cu}^\sim(A):= \{ [a] - n[1_A] \mid [a] \in \mathrm{Cu}(A),\ n \in \mathbb{N} \}. \notag
\end{align}
This set carries an order by declaring that $[a]-n[1_A] \leq [b] - m[1_A]$ holds in $\mathrm{Cu}^\sim(A)$ if there is $k \in \mathbb{N}$ such that $[a]+m[1_A]+k[1_A] \leq [b] + n [1_A] + k [1_A]$ in $\mathrm{Cu}(A)$.

Now assume $A$ is non-unital and $\pi: A^\dagger \to \mathbb{C}$ is the canonical quotient map.
The augmented Cuntz semigroup of $A$ is then defined as the subsemigroup of $\mathrm{Cu}^\sim(A^\dagger)$ given by 
\begin{align}
	\mathrm{Cu}^\sim (A) := \{ [a] - n[1_{A^\dagger} ] \mid [a] \in \mathrm{Cu}({A^\dagger}),\ \mathrm{Cu}(\pi) ([a]) = n \}. \notag
\end{align}
We endow $\mathrm{Cu}^\sim(A)$ with the order coming from $\mathrm{Cu}^\sim (A^\dagger)$.
We refer to \cite{Rob12,RS2021} for more details about this construction.
As we shall see below, the augmented Cuntz semigroup is a powerful tool to classify maps between certain classes of $\C$-algebras. 

When a $\C$-algebra is simple, exact, $\Z$-stable and admits non-trivial traces, $\mathrm{Cu}^\sim(A)$ can be calculated using its K-theory and tracial data (\cite[Theorem 6.11]{RS2021}).
Indeed, if $\overline{\mathbb{R}} := \mathbb{R} \cup \{\infty\}$ and $\mathrm{Lsc}(T^+(A), \overline{\mathbb{R}})$ denotes the lower semicontinuous functions $T^+(A) \to \overline{\mathbb{R}}$ that are linear and map the zero trace to $0$,
one always has a natural isomorphism 
\begin{align}\label{CuSimFormula}
\mathrm{Cu}^\sim (A) \cong K_0(A) \sqcup \mathrm{Lsc}(T^+(A), \overline{\mathbb{R}}).
\end{align}
The ordered semigroup structure on the right hand side is given as follows. 
The sets $K_0(A)$ and $\mathrm{Lsc}(T^+(A), \overline{\mathbb{R}}) $ are considered disjoint and each set is separately endowed with its natural order and addition.
For $x \in K_0(A)$, we consider the function $\hat{x}: T^+(A) \to \mathbb{R}\subset\overline{\mathbb{R}}$ given by evaluation on $x$. 
For any $f \in \mathrm{Lsc}(T^+(A), \overline{\mathbb{R}})$, the addition operation of mixed terms is given via $x+f:= \hat{x} + f \in \mathrm{Lsc}(T^+(A))$.
The order is defined by declaring that $f \leq x$ holds if $f \leq \hat{x}$ in $\mathrm{Lsc}(T^+(A), \overline{\mathbb{R}})$, and $x \leq f$ if there is a strictly positive function $h \in \mathrm{Lsc}(T^+(A), \overline{\mathbb{R}})$ such that $\hat{x}+h = f$.
We refer to \cite[Section 6.3]{RS2021} for more details about the above isomorphism. 

A major tool for this paper is a classification result by Robert that applies to $^*$-homomorphisms  
where the domain is an inductive limit of \emph{1-dimensional noncommutative CW complexes} (henceforth abbreviated as 1-NCCW complexes) with vanishing $K_1$-groups and the codomain has \emph{stable rank one} (i.e., invertible elements are dense in its minimal unitisation).
The class of 1-NCCW complexes was introduced by Eilers--Loring--Pedersen in \cite{Eilers1998}.
These algebras are defined as pullback $\C$-algebras of the form $C([0,1], F) \oplus_{F\oplus F} E$ with the linking morphism $C([0,1], F) \to F \oplus F$ given by $\mathrm{ev}_0 \oplus \mathrm{ev}_1$.

\begin{theorem}[{\cite[Theorem 1.0.1]{Rob12}}]\label{thm:Robert}
	Let $A$ be either a $1$-NCCW complex with trivial $K_1$-group, or a sequential inductive limit of such $\C$-algebras, or a $\C$-algebra stably isomorphic to one such inductive limit.
	Let $B$ be a $\C$-algebra with stable rank one.
	Then for every $\mathrm{Cu}$-morphism
	\begin{equation}
		\alpha: \mathrm{Cu}^\sim(A) \to \mathrm{Cu}^\sim (B) \notag
	\end{equation}
	such that $\alpha([s_A]) \leq [s_B]$, where $s_A \in A_+$ and $s_B \in B_+$ are strictly positive elements, there exists a $^*$-homomorphism
	\begin{equation}
		\varphi: A \to B \notag
	\end{equation}
	such that $\mathrm{Cu}^\sim (\varphi) = \alpha$.
	Moreover, $\varphi$ is unique up to approximate unitary equivalence.
\end{theorem}

We shall henceforth refer to the class of \C-algebras that satisfy the assumptions of the above theorem in place of $A$ as \emph{Robert's class}.
For subsequent applications of the theorem we note that every \C-algebra belonging to Robert's class also has stable rank one.

\subsection{The Razak--Jacelon algebra $\W$ and the $\C$-algebra $\Zo$}\label{section:WandZ0}

Among the class of inductive limits of 1-NCCW complexes with vanishing $K_1$-groups, there are two important examples with remarkable properties that we discuss below.

The \emph{Razak--Jacelon algebra} $\W$ is the unique algebraically simple nuclear stably projectionless $\Z$-stable monotracial and $KK$-contractible $\C$-algebra up to isomorphism. 
It is constructed as an inductive limit of certain 1-NCCW complexes known as Razak building blocks \cite{Razak2002, Jacelon2013}. 
This algebra absorbs tensorially the universal UHF algebra and hence it is also $\Z$-stable. 
Via classification results (\cite[Corollary 6.7]{Elliott2020}), it is known that it is self-absorbing, i.e., $\W \otimes \W \cong \W$.
Furthermore, by the Kirchberg--Phillips classification theorem, one sees also that $\W \otimes \mathcal{O}_\infty \cong \mathcal{O}_2 \otimes \K$.
The $\C$-algebra $\W$ can be regarded as the stably finite analogue of $\mathcal{O}_2$.

The Razak--Jacelon algebra belongs to the broader class of stably projectionless simple nuclear $\Z$-stable $KK$-contractible $\C$-algebras.
This class was classified by Elliott--Gong--Lin--Niu (\cite[Theorem 7.5]{Elliott2020}).
(Note that \cite{CE} guarantees that the assumption of $\Z$-stability agrees with the assumption of finite nuclear dimension appearing in this reference.)
In particular, this applies to simple stably projectionless nuclear $\C$-algebras that absorb tensorially the Razak--Jacelon algebra $\W$.
An important result for this note is that $^*$-homomorphisms between $KK$-contractible $\C$-algebras are classified by their tracial behaviour.

\begin{theorem}[{\cite[Theorem 6.3]{Sza21}, \cite{Elliott2020}}]\label{thm:class.maps.KK0}
	Let $A$ and $B$ be simple separable stably projectionless nuclear $\Z$-stable $\C$-algebras such that $KK(A,A)= KK(B,B)=0$. Let $\psi, \varphi: A \to B$ be $^*$-homomorphisms. Then $\varphi$ and $\psi$ are approximately unitarily equivalent if and only if $\tau \circ \psi = \tau \circ \varphi$ for all $\tau \in T^+(B)$.
\end{theorem}

On the other hand, the $\C$-algebra $\Zo$ is a stably projectionless $\C$-algebra that is also an inductive limit of 1-NCCW complexes with vanishing $K_1$-groups.
It is the unique (up to isomorphism) algebraically simple and monotracial \C-algebra in Robert's class with $K_0(\Zo) = \mathbb{Z}$.
(As it belongs to Robert's class, one automatically has $K_1(\Zo) = 0$.)
It has interesting and useful properties like being self-absorbing and $\Z$-stable (\cite[Definition 8.1, Corollary 13.4]{GL20}) and it can also be regarded as a stably projectionless analog of $\Z$. 
In \cite[Theorem 15.8]{GL20}, it was shown that separable, simple, nuclear and $\Zo$-stable $\C$-algebras satisfying the UCT are classified by the Elliott invariant, which in this case obeys the condition that the pairing between traces and the $K_0$-group has to be trivial.

The $\C$-algebras $\W$ and $\Zo$ both belong to Robert's class.
As an application, one can produce a useful interplay between these algebras by constructing special $^*$-homomorphisms $\Zo \to \W$ and $\W \to \Zo$.
To be more precise, equation \eqref{CuSimFormula} yields that
\begin{equation} \label{eq:aCu-for-Zo-and-W}
\mathrm{Cu}^\sim(\Zo) \cong \mathbb{Z} \sqcup \overline{\mathbb{R}} \quad \text{and} \quad \mathrm{Cu}^\sim(\W) \cong \{0\}\sqcup\overline{\mathbb{R}}.
\end{equation}
One can then define order preserving maps between the augmented Cuntz semigroups in the following way:
the map $\W \to \Zo$ is induced by the natural inclusion $\mathrm{Cu}^\sim(\W) \hookrightarrow \mathrm{Cu}^\sim(\Zo)$ and the $^*$-homomorphism $\Zo \to \W$ is induced by the map that is equal to the identity on $\overline{\mathbb{R}}$ and sends $\mathbb{Z}$ to $0$.
It follows from the construction of these maps that they vanish in $K_0$ and preserve the corresponding tracial state.

\begin{theorem}[{cf.\ \cite[Definition 8.12]{GL20}}]\label{thm:useful.maps}
	There exist unique trace preserving $^*$-homomorphisms $\varphi_{\Zo}: \W \to \Zo$ and $\varphi_{\W}: \Zo \to \W$ (up to approximate unitary equivalence). In particular, 
	\begin{equation}
		K_0(\varphi_{\W}) = 0 \qquad \text{and} \qquad K_0(\varphi_{\Zo})=0. \notag
	\end{equation}
\end{theorem}

There is another useful automorphism of $\Zo$ that we will use in the next section. This automorphism is obtained from defining a map $\Lambda$ at the level of $\mathrm{Cu}^\sim(\Zo)$ that sends $n$ to $-n$ on $\mathbb{Z}$ and agrees with the identity map on $\overline{\mathbb{R}}$.
By Robert's classification theorem, there exists a $^*$-endomorphism $\tilde{\sigma}: \Zo \to \Zo$ that induces $\Lambda$.
Since by the same theorem one has that $\tilde{\sigma}^2$ is approximately inner, it follows from the Elliott intertwining argument that $\tilde{\sigma}$ is approximately unitarily equivalent to an automorphism; see \cite[Corollary 2.3.4]{Rordam}.

\begin{theorem}[{cf.\ \cite[Definition 8.13]{GL20}}]
	There exists a unique trace preserving automorphism $\sigma: \Zo \to \Zo$ (up to approximate unitary equivalence) such that $K_0(\sigma) = -\mathrm{id}_{K_0(\Zo)}$.
\end{theorem}

We include the statement of the following lemma proved in \cite{GL20}, which is another application of Robert's classification theorem.

\begin{lemma}[{\cite[Lemma 8.14]{GL20}}]\label{lem:GL}
	Consider the $^*$-homomorphisms $\Upsilon, \Omega: \Zo \to M_2(\Zo)$ given by
	\begin{equation}\label{lem:GL.eq1}
		\Upsilon := \mathrm{id}_{\Zo} \oplus \sigma \qquad \text{and} \qquad \Omega := (\varphi_{\Zo} \circ \varphi_{\W}) \otimes 1_{M_2}.
	\end{equation}
	Then $\Upsilon$ is approximately unitarily equivalent to $\Omega$. 
\end{lemma}

\section{The Main result}

In this section we prove the main result of this note. 
We begin by introducing some notation that will be used throughout. 
Given $n \in \mathbb{N}$ and $j \in \{1, 2, \ldots, 2n+1\}$, we consider the $^*$-homomorphism
\begin{align}\label{ji}
	\iota_j^{A}: A \to M_{2n+1}(A) \quad \text{by} \quad \iota_j^A(a) = e_{ii}\otimes a.
\end{align} 
For $j \in \{1, \ldots, 2n\}$, we consider $\kappa_j^{A}: M_2(A) \to M_{2n+1} (A)$ by
\begin{align}\label{ki}
	\kappa_j^{A}  
	\left( \left( a_{k\ell} \right)_{k,\ell=0,1} \right)
	=  \sum_{k,\ell=0,1} e_{j+k,j+\ell} \otimes a_{k\ell}. 
\end{align}     
We will typically omit the superscript $A$ when  it is clear from context which algebra we are referring to. 

We note that by the classification theorem \cite[Theorem 15.8]{Elliott2020}, it follows that $\Zo$ is isomorphic to $M_n (\Zo)$ for all $n \in \mathbb{N}$.
We say that an isomorphism $\Phi: \Zo\to M_n(\Zo)$ is \emph{$K$-positive}, if $K_0(\Phi)=K_0(\iota_1)$.
Note that due to the existence of the aforementioned automorphism $\sigma$ on $\Zo$, $K$-positivity is not automatic in this context.

\begin{lemma}\label{lem:Gamma}
	Let $n \in \mathbb{N}$. 
	Consider the $^*$-homomorphism $\Gamma_n:  \Zo \to M_{2n+1}(\Zo)$ given by 
	\begin{align}\label{Gamman}
	\Gamma_n(a) := 
	\sum\limits_{j=0}^n \iota_{2j+1}(a)+ \sum\limits_{j=1}^n \iota_{2j}(\sigma (a)). \notag
	\end{align}
	Then any $K$-positive isomorphism $\Zo \to M_{2n+1}(\Zo)$ is approximately unitarily equivalent to $\Gamma_n$.
\end{lemma}

\begin{proof}
	Notice that $\Gamma_n(a)$ is a matrix of the form
	\begin{equation}
		\Gamma_n(a) = \begin{pmatrix}
			a &  & & &  \\
			& \sigma(a) &  \\
			& & a \\
			& & & \sigma(a) & \\
			& & & & \ddots \\
			& & & & & a
		\end{pmatrix}. \notag
	\end{equation}
	Since $K_0(\iota_k) = K_0(\iota_0)$ for all $k\leq 2n+1$ and $K_0(\sigma) = - \id_{K_0(\Zo)}$, we see that
	\begin{align}
		K_0(\Gamma_n) & =  K_0\left( \sum_{j=0}^n \iota_{2j+1} \circ \id_{\Zo} + \sum_{j=1}^n \iota_{2j}\circ \sigma \right)  \notag \\
		& =  \sum_{j=0}^n K_0(\iota_1) - \sum_{j=1}^n K_0(\iota_1) \notag \\
		& =  K_0(\iota_1). \notag
	\end{align}
	On the other hand, the uniqueness of the tracial states $\tau_{\Zo}$ and $\tr_{2n+1}$ on $\Zo$ and $M_{2n+1}(\mathbb{C})$, respectively, yields that the unique trace on $M_{2n+1}(\Zo)$ is of the form $\tr_{2n+1} \otimes \tau_{\Zo}$. Hence, using that $\sigma$ is a trace-preserving automorphism (i.e.\ $\tau_{\Zo} \circ \sigma = \tau_{\Zo}$), we obtain
	\begin{align}
		(\tr_{2n+1} \otimes \tau_{\Zo})\circ \Gamma_n   &=  \tau_{\Zo}. \notag
	\end{align}
	If we keep in mind \eqref{eq:aCu-for-Zo-and-W}, the above implies that $\mathrm{Cu}^\sim(\Gamma_n) = \mathrm{Cu}^\sim(\Phi)$ for any $K$-positive isomorphism $\Phi: \Zo\to M_{2n+1}(\Zo)$. 
	By Robert's Theorem~\ref{thm:Robert}, $\Gamma_n$ is thus approximately unitarily equivalent to $\Phi$.
\end{proof}

\begin{lemma}\label{lem:Ad(u)}
	Let $A$ be a $\sigma$-unital $\C$-algebra and assume that a unitary $u \in \mathcal{U}(\mathcal{M}(A))$ homotopic to $1_{\mathcal{M}(A)}$.
	Then there exists a sequence $(v_n)_{n\in\mathbb{N}} \subset \mathcal{U}(1+A)$ such that $(\mathrm{Ad}(v_n))_{n\in\mathbb{N}}$ converges to $\mathrm{Ad}(u)$ in the point-norm topology.
\end{lemma}

\begin{proof}
	This statement is a consequence of a special case of \cite[Lemma 4.3]{GS22}, applied to $D = 0$, the trivial action in place of $\beta$, and for the \C-algebra $A$ in place of $B$.
\end{proof}

We now proceed with the key technical lemma of this note, which is a non-unital and stably finite version of the reduction argument for homotopic maps in \cite[Lemma 5.10]{Sza21} (this argument originates in \cite{Phillips1997}).
We will use the following standard notation:
given a $^*$-homomorphism $\Phi: A \to C([0,1],B)$ and $t \in [0,1]$, we write $\Phi_t(a):= \Phi(a)(t)$.
The basic idea of the proof will be to use the maps $\varphi_{\W}$ and $\varphi_{\Zo}$ (see Theorem \ref{thm:useful.maps}) to move a given unitary equivalence of $\Phi_s \otimes \id_{\W}$ and $\Phi_t \otimes \id_{\W}$ from $B \otimes \W$ to $B \otimes \Zo$ combined with a handy identification of $\Zo$ with $M_{2n+1}(\Zo)$.

\begin{lemma}\label{keylemma}
	Let $A$ and $B$ be separable $\C$-algebras.
Let $\Phi: A \to C([0,1], B)$ be a $^*$-homomorphism.
Suppose that 
	\begin{align}
		\Phi_s \otimes \id_{\W} \approx_{\mathrm{u}} \Phi_t \otimes \id_{\W} \quad \text{for all} \ s,t \in [0,1]. \notag
	\end{align}
	Then  
	\begin{align} \label{eq:key-lemma-conclusion}
		\Phi_0 \otimes \id_{\Zo} \approx_{\mathrm{u}} \Phi_1 \otimes \id_{\Zo}.
	\end{align}
\end{lemma}

\begin{proof}
Let $\mathfrak{F} \subset A \otimes \Zo$ be a finite set and $\varepsilon >0$. 
We want to find a unitary in $B\otimes\Zo$ that $(\mathfrak{F},\varepsilon)$-approximately conjugates the first map in \eqref{eq:key-lemma-conclusion} onto the second.
Without loss of generality, we may assume that there exist finite sets of contractions $\mathfrak{F}_A\subset A$ and $\mathfrak{F}_{\Zo}\subset\Zo$ such that
\[
\mathfrak{F} = \left\{ a\otimes z \mid a\in\mathfrak{F}_A,\ z \in \mathfrak{F} \right\}.
\]
By uniform continuity, there exists $n \in \mathbb{N}$ such that if $s,t\in [0,1]$ are arbitrary parameters with $|s-t| \leq \frac{1}{n}$ then 
\begin{align}\label{kl:eq1}
	\Phi_s(a) & \approx\bsub{\varepsilon/9} \Phi_t(a), \quad a \in \mathfrak{F}_A.
\end{align}
By assumption, one has for each $j \in \{1, \ldots, n\}$ a unitary $v_j \in \mathcal{U}(1+B\otimes \W)$ such that
\begin{align}
	v_j \left(\Phi_0 (a) \otimes \varphi_\W(z) \right) v_j^* & \approx\bsub{\varepsilon/9} \Phi_{\frac{j}{n}} (a)\otimes \varphi_\W (z) \notag
\end{align}
for all $a \in \mathfrak{F}_A$ and $z \in \mathfrak{F}_{\Zo}$.
Similarly, there are unitaries $u_j \in \mathcal{U}(1+B\otimes \W)$ for $j \in \{0, 1, \ldots, n-1\}$ such that
\begin{align}\label{kl:eq3}
	u_j \left(\Phi_{\frac{j}{n}}(a) \otimes \varphi_\W(z) \right)u_j^* & \approx\bsub{\varepsilon/9}  \Phi_1(a) \otimes \varphi_\W (z), 
\end{align}
for all $a \in \mathfrak{F}_A$ and $z \in \mathfrak{F}_{\Zo}$.

We will use the maps $\varphi_{\W}: \Zo \to \W$, $\varphi_{\Zo}: \W \to \Zo$, $\sigma: \Zo \to \Zo$ and $\Omega, \Upsilon: \Zo \to M_2(\Zo)$ introduced in Section \ref{section:WandZ0}.
By Lemma \ref{lem:GL}, the maps $\Omega$ and $\Upsilon$ are approximately unitarily equivalent.
Hence, there is a unitary $\mathtt{u} \in \mathcal{U}(1+M_2(\Zo))$ such that
\begin{align}\label{kl:eq4}
	\mathtt{u} \Omega(z) \mathtt{u}^* & \approx\bsub{\varepsilon/9}  \Upsilon(z), \qquad z \in \mathfrak{F}_{\Zo}.
\end{align}
Consider the unitary
\begin{align}
	\begin{pmatrix}
		0 & 1_{\mathcal{M}(\Zo)} \\
		1_{\mathcal{M}(\Zo)} & 0
	\end{pmatrix} \in M_2(\mathcal{M}(\Zo)), \notag
\end{align}
which is homotopic to $1_{M_2(\mathcal{M}(\Zo))}$. 
We set $\Upsilon': \Zo \to M_2(\Zo)$ by
\begin{align}\label{kl:eq15}
	\Upsilon'(z) := \mathrm{Ad}
		\begin{pmatrix}
		0 & 1_{\mathcal{M}(\Zo)} \\
		1_{\mathcal{M}(\Zo)} & 0
		\end{pmatrix} 
		 \circ \Upsilon (z) = \begin{pmatrix}
		\sigma(a) & 0 \\
		0 & a 
	\end{pmatrix}.
\end{align}
It follows from Lemma \ref{lem:Ad(u)} 
that
$\Upsilon \approx_{\mathrm{u}} \Upsilon'$.
The approximate unitary equivalence between $\Omega$ and $\Upsilon$ entails
$\Omega \approx_{\mathrm{u}} \Upsilon'$.
Then, there exists a unitary $\mathtt{v} \in \mathcal{U}(1 + M_2(\Zo))$ such that
\begin{align}\label{kl:eq5}
	\mathtt{v} \Omega(z) \mathtt{v}^* \approx\bsub{\varepsilon/9} \Upsilon'(z).
\end{align}
We define $\theta, \theta': B \otimes \W \to B \otimes M_2(\Zo)$ by
\begin{align}
	\theta := \id_B \otimes (\mathrm{Ad}(\mathtt{u})\circ (\varphi_{\Zo} \otimes 1_{M_2})) \notag
\end{align}
and
\begin{align} 
	\theta' := \id_B \otimes (\mathrm{Ad}(\mathtt{v})\circ (\varphi_{\Zo} \otimes 1_{M_2})). \notag
\end{align}
In particular, these yield
\begin{align}\label{kl:eq16}
	\theta(b \otimes \varphi_\W(z)) = b \otimes \mathtt{u} \Omega(z)\mathtt{u}^*  \overset{\eqref{kl:eq4}}{\approx}_{\makebox[0pt]{\ \footnotesize$\varepsilon/9$}} \hspace{3mm} b\otimes \Upsilon(z) 
\end{align}
and
\begin{align}
		\theta'(b \otimes \varphi_\W (z))  = b \otimes \mathtt{v} \Omega(z) \mathtt{v}^*  \overset{\eqref{kl:eq5}}{\approx}_{\makebox[0pt]{\ \footnotesize$\varepsilon/9$}} \hspace{3mm} b \otimes \Upsilon'(z) \notag
\end{align}
for $z \in \mathfrak{F}_{\Zo}$ and any contraction $b \in B$.
Therefore
\begin{align}\label{kl:eq17}
	\theta(u_j) \left( \Phi_{\frac{j}{n}}(a) \otimes \Upsilon(z) \right) \theta(u_j)^* &
	\overset{\eqref{kl:eq16}}{\approx}_{\makebox[0pt]{\ \footnotesize$\varepsilon/9$}} \hspace{3mm} \theta(u_j) \theta \left( \Phi_{\frac{j}{n}}(a) \otimes \varphi_\W (z) \right) \theta(u_j)^* \notag \\
	& \overset{\phantom{\eqref{kl:eq3}}}{=} \hspace{3mm} \theta \left( u_j \biggl( \Phi_{\frac{j}{n}}(a) \otimes \varphi_\W (z)\biggr)u_j^* \right) \notag \\
	& \overset{\eqref{kl:eq3}}{\approx}_{\makebox[0pt]{\ \footnotesize$\varepsilon/9$}} \hspace{3mm} \theta(\Phi_1(a) \otimes \varphi_\W(z)) \notag \\
	& \overset{\eqref{kl:eq16}}{\approx}_{\makebox[0pt]{\ \footnotesize$\varepsilon/9$}} \hspace{3mm} \Phi_1(a) \otimes \Upsilon(z),
\end{align}
for $j \in \{0, \ldots, n-1 \}$.
A similar calculations shows
\begin{align}\label{kl:eq18}
	\theta'(v_j) \left( \Phi_{0} (a) \otimes \Upsilon'(z) \right) \theta'(v_j)^* \approx\bsub{\varepsilon/3} \Phi_{\frac{j}{n}}(a) \otimes \Upsilon'(z) 
\end{align}
for $j \in \{1, \ldots, n \}$.

With the notation introduced in \eqref{ji} and \eqref{ki}, we define unitaries in $M_{2n+1}(B\otimes \Zo)$ in the following way:
\begin{align}
	V &:= e_{11}\otimes 1_{\mathcal{M}(B \otimes \Zo)} + \sum_{j=1}^n \kappa_{2j}(\theta'(v_j)), \notag \\
	U &:= \sum_{j=1}^{n} \kappa_{2j-1}(\theta(u_{j-1})) + e_{2n+1,2n+1}\otimes 1_{\mathcal{M}(B\otimes \Zo)}. \notag
\end{align}
Schematically, these unitaries correspond to block-diagonal matrices of the form
\begin{align}
	V = \begin{pmatrix}
		1_{\mathcal{M}(B \otimes \Zo)}  \\
		& \theta'(v_1) \\ 
		& & \theta'(v_2) \\
		& & & \ddots \\
		& & & & \theta'(v_n)
	\end{pmatrix} \notag
\end{align}
and
\begin{align}
	U = \begin{pmatrix}
		\theta(u_0) \\ 
	 	& \theta(u_1) \\
		& & \ddots \\
		& & & \theta(u_{n-1}) \\	
		& & & & 1_{\mathcal{M}(B \otimes \Zo)} 	
	\end{pmatrix}. \notag
\end{align}
If $b \in B \otimes \Zo$ and $x_0, x_1, \ldots, x_n \in M_2(B\otimes \Zo)$, then
\begin{align}\label{kl:eq12}
V\left(\iota_1(b) + \sum_{j=1}^n \kappa_{2j} (x_j)\right) V^*  = \iota_1 (b) + \sum_{j=1}^n \kappa_{2j}(\theta'(v_j) x_j \theta'(v_j)^* ) 
\end{align}
and
\begin{align}\label{kl:eq13}
&U \left( \sum_{j=1}^{n} \kappa_{2j-1}(x_{j-1})  + \iota_{2n+1}(b)   \right)	U^* \notag \\  & \hspace{2cm} = \sum_{j=1}^{n} \kappa_{2j-1}(\theta(u_{j-1})x_{j-1} \theta(u_{j-1})^*) + \iota_{2n+1}(b).
\end{align}
We will now employ the $^*$-homomorphism $\Gamma_n:  \Zo \to M_{2n+1}(\Zo)$ given in Lemma \ref{lem:Gamma} as
\begin{align}
	\Gamma_n(a) := 
	\sum\limits_{j=0}^n \iota_{2j+1}(a)+ \sum\limits_{j=1}^n \iota_{2j}(\sigma (a)). \notag
\end{align}
Observe that we can write $\Gamma_n$ using either $\Upsilon$ or $\Upsilon'$  in the following way
\begin{align}\label{kl:eq21}
	\Gamma_n(a) = 
\begin{pmatrix}
	\Upsilon(a)  \\
	& \ddots \\
	& & \Upsilon(a) \\
	& & & a
\end{pmatrix} =
\begin{pmatrix}
	a \\
	& \Upsilon'(a) \\
	& & \ddots \\
	& & & \Upsilon'(a)
\end{pmatrix}. 
\end{align}
Having this observation mind, we then obtain for $a \in \mathfrak{F}_A$ and $z \in \mathfrak{F}_{\Zo}$ that
\begin{align}\label{kl:eq11}
	V & ( \Phi_0(a)  \otimes \Gamma_n(z) )V^* \notag \\ 
	& \overset{\phantom{\eqref{kl:eq15}}}{=} \hspace{3mm} V \left(	\iota_1(\Phi_0(a) \otimes z ) + \sum_{j=1}^n \kappa_{2j} \left(\Phi_0(a) \otimes \Upsilon'(z) \right) \right)V^* \notag 
	\\
	& \overset{\eqref{kl:eq12}}{=}  \hspace{3mm} \iota_1 (\Phi_0(a)\otimes z ) + \sum_{j=1}^n \kappa_{2j}\left(\theta'(v_j) \left( \Phi_0(a) \otimes 
	\Upsilon' (z) \right) \theta'(v_j^* ) \right) \notag  
	\\
	& \overset{\eqref{kl:eq18}}{\approx}_{\makebox[0pt]{\footnotesize$\varepsilon/3$}} \hspace{3mm} \iota_1 (\Phi_0(a)\otimes z) + \sum_{j=1}^n \kappa_{2j}\left(\Phi_{\frac{j}{n}}(a) \otimes \Upsilon'(z) \right) \notag 
	\\
	& \overset{\eqref{kl:eq15}}{=} \hspace{3mm} \iota_1(\Phi_0(a) \otimes z) + \sum_{j=1}^{n} \left(  \iota_{2j}\left(\Phi_{\frac{j}{n}}(a) \otimes \sigma(z)\right) + \iota_{2j+1}\left(\Phi_{\frac{j}{n}} (a) \otimes z\right) \right) \notag 
	\\
	& \overset{\eqref{kl:eq1}}{\approx}_{\makebox[0pt]{\footnotesize \hspace{1mm}$\varepsilon/9$}} \hspace{3mm}  \iota_1(\Phi_0(a) \otimes z) + \sum_{j=1}^n \left( \iota_{2j} \left(\Phi_{\frac{j-1}{n}}(a) \otimes \sigma(z) \right) + \iota_{2j+1} \left(\Phi_{\frac{j}{n}} (a) \otimes z \right) \right) \notag
	\\
	& \overset{\phantom{\eqref{kl:eq15}}}{=} \hspace{3mm}  \sum_{j=1}^n \left( \iota_{2j-1} \left(\Phi_{\frac{j-1}{n}}(a) \otimes z \right) + \iota_{2j} \left(\Phi_{\frac{j-1}{n}}(a) \otimes \sigma(z) \right) \right) + \iota_{2n+1}(\Phi_1(a) \otimes z)  \notag 
	\\ 
	& \overset{\eqref{lem:GL.eq1}}{=} \hspace{3mm} \sum_{j=1}^{n} \kappa_{2j-1} \left(\Phi_{\frac{j-1}{n}} (a) \otimes \Upsilon (z) \right) + \iota_{2n+1} ( \Phi_1(a) \otimes z). 
\end{align}
Hence
\begin{align}\label{kl:eq20}
	& U  V (\Phi_0(a) \otimes  \Gamma_n(z))V^*U^* \notag \\
	& \overset{\eqref{kl:eq11}}{\approx}_{\phantom{.}\makebox[0pt]{\footnotesize$4\varepsilon/9$}} \hspace{4mm} U\left(\sum_{j=1}^{n} \kappa_{2j-1} \left(\Phi_{\frac{j-1}{n}} (a) \otimes \Upsilon (z) \right) + \iota_{2n+1} ( \Phi_1(a) \otimes z)\right)U^* \notag \\
	& \overset{\eqref{kl:eq13}}{=} \hspace{3mm} \sum_{j=1}^{n} \kappa_{2j-1} \left( \theta(u_{j-1}) \left( \Phi_{\frac{j-1}{n}}(a) \otimes \varphi_\W (z)\right) \theta(u_{j-1})^*  \right) + \iota_{2n+1} (\Phi_1(a) \otimes z)  \notag \\
	& \overset{\eqref{kl:eq17}}{\approx}_{\makebox[0pt]{\footnotesize$\varepsilon/3$}} \hspace{3mm} \sum_{j=1}^n \kappa_{2j-1}(\Phi_1(a) \otimes \Upsilon(z)) + \iota_{2n+1}(\Phi_1(a) \otimes z)  \notag \\
	& \overset{\eqref{kl:eq21}}{=} \hspace{3mm}\Phi_1(a) \otimes \Gamma_n(z). 
\end{align}
Let $\gamma: \Zo \to M_{2n+1}(\Zo)$ be any $K$-positive isomorphism.
By Lemma \ref{lem:Gamma}, $\gamma$ and $\Gamma_n$ are approximately unitarily equivalent.
So there is a unitary $W \in \mathcal{U}(1+M_{2n+1}(\Zo))$ such that
\begin{align}\label{kl:eq8}
	W \gamma(z) W^* \approx\bsub{\varepsilon/9} \Gamma_n(z),  \qquad z \in \mathfrak{F}_{\Zo}.
\end{align}
Let us consider the isomorphism $\eta: B \otimes \Zo \to B \otimes M_{2n+1}(\Zo)$ given by $\eta := \id_B \otimes \left(\mathrm{Ad}(W) \circ \gamma \right)$, and set
\begin{align}\label{kl:eq9}
	w := \eta^{-1} (UV) \in \mathcal{U}(1+B\otimes \Zo). 
\end{align}
Finally, we observe that 
\begin{align}
	 w(\Phi_0(a) \otimes z) w^* & \overset{\eqref{kl:eq9}}{=} \hspace{5mm} \eta^{-1} \left(UV  \eta\left( \Phi_0(a)\otimes z \right)  V^*U^* \right) \notag \\
	 & \overset{\phantom{\eqref{kl:eq11}}}{=} \hspace{5mm} \eta^{-1} \left( UV \left(\Phi_0(a)\otimes W \gamma(z) W^* \right) V^*U^* \right) \notag \\
	 & \overset{\eqref{kl:eq8}}{\approx}_{\makebox[0pt]{\footnotesize$\varepsilon/9$}} \hspace{5mm} \eta^{-1}\left(UV \left( \Phi_0(a) \otimes \Gamma_n(z) \right) V^*U^* \right) \notag \\
	 & \overset{\eqref{kl:eq20}}{\approx}_{\makebox[0pt]{\phantom{1}\footnotesize$7\varepsilon/9$}} \hspace{5mm} \eta^{-1} \left( \Phi_1(a) \otimes \Gamma_n(z) \right) \notag \\
	 & \overset{\eqref{kl:eq8}}{\approx}_{\makebox[0pt]{\footnotesize$\varepsilon/9$}} \hspace{5mm} \eta^{-1} \left(\Phi_1(a) \otimes W \gamma(z) W^* \right)\notag \\
	 & \overset{\phantom{\eqref{kl:eq11}}}{=} \hspace{5mm} \eta^{-1} \left(\eta (\Phi_1(a) \otimes z) \right) \notag  \\
	 & \overset{\phantom{\eqref{kl:eq11}}}{=} \hspace{5mm} \Phi_1(a) \otimes z. \notag
\end{align}
This shows that $\Phi_1(a) \otimes z \approx\bsub{\varepsilon} w (\Phi_0(a)\otimes z)w^*$ for all $a \in \mathfrak{F}_A$ and $\mathfrak{F}_{\Zo}$.
This verifies $\Phi_0 \otimes \id_{\Zo} \approx_{\mathrm{u}} \Phi_1 \otimes \id_{\Zo}$.
\end{proof}

We now present the notion of trace-preserving homotopy for $^*$-homomorphisms between certain $\C$-algebras.
It will be a key ingredient in the main result of this paper.

\begin{definition}\label{def:tracial.homotopy}
	Let $\varphi, \psi: A \to B$ be $^*$-homomorphisms between $\C$-algebras with $T^+(B)\neq\emptyset$.
	We say that $\varphi$ and $\psi$ are \emph{trace-preservingly homotopic} if there is a $^*$-homomorphism $\Phi: A \to C([0,1], B)$ with $\Phi_0 = \varphi$, $\Phi_1 = \psi$ such that 
	$$\tau(\Phi_t(a)) = \tau(\Phi_s(a))$$ 
	for all $a \in A$, $s,t \in [0,1]$ and $\tau \in T^+(B)$. 
\end{definition}

We now state a homotopy rigidity result for $^*$-homomorphisms, which can be viewed as the main result of this paper.
The key feature of it, as pointed out in the introduction, is that it does not assume that any of the underlying \C-algebras has to satisfy the UCT.

\begin{theorem}\label{maintheorem}
	Let $A$ and $B$ be separable, simple and nuclear $\C$-algebras with $T^+(B)\neq\emptyset$.
	Suppose $\varphi, \psi: A \to B$ are trace-preservingly homotopic $^*$-homomorphisms.
	Then
	\begin{equation}
		\varphi \otimes \id_{\Zo} \approx_{\mathrm{u}} \psi \otimes \id_{\Zo}. \notag 
	\end{equation} 
\end{theorem}

\begin{proof}
	By hypothesis there exists a $^*$-homomorphism $\Phi: A \to C([0,1], B)$ such that $\Phi_0 = \varphi, \Phi_1 = \psi$ and 
	\begin{equation}
		\theta\circ\Phi_t = \theta\circ\Phi_s, \quad \theta \in T^+(B), \ s,t \in [0,1]. \notag
	\end{equation} 
	This implies
	\begin{equation}
		\tau \circ (\Phi_t \otimes \id_{\W}) = \tau \circ (\Phi_s \otimes \id_{\W}), \quad \tau \in T^+(B \otimes \W),\ s,t \in [0,1]. \notag
	\end{equation}
	On the other hand, it follows that $A \otimes \W$ and $B \otimes \W$ are stably projectionless and $\Z$-stable (\cite[Corollary 6.7]{Elliott2020}).
	Then, by Theorem \ref{thm:class.maps.KK0}, we obtain that $\Phi_t \otimes \id_{\W} \approx_{\mathrm{u}} \Phi_s \otimes \id_{\W}$ for all $s,t \in [0,1]$.
	Finally, by Lemma \ref{keylemma}, the $^*$-homomorphisms $\varphi \otimes \id_{\Zo}$ and $\psi \otimes \id_{\Zo}$ are approximately unitarily equivalent.
\end{proof}

As a byproduct of Theorem \ref{maintheorem}, we obtain a precursor $\Zo$-stable classification theorem of sorts for separable, simple, nuclear $\C$-algebras without requiring the UCT.
For this result, we will need to work with the notion of trace-preserving homotopy equivalence for $\C$-algebras induced from the notion above.

\begin{definition}
	Let $A$ and $B$ be $\C$-algebras with $T^+(A)\neq\emptyset\neq T^+(B)$.
	We say that $A$ and $B$ are \emph{trace-preservingly homotopy equivalent} if there exist $^*$-homomorphisms $\varphi: A \to B$ and $\psi: B \to A$ such that $\psi \circ \varphi$ and $\varphi \circ \psi$ are trace-preservingly homotopic to $\id_A$ and $\id_B$, respectively. 
\end{definition}

\begin{theorem}\label{theo:class}
	Let $A$ and $B$ be separable, simple, nuclear $\C$-algebras.
	If $A$ and $B$ are trace-preservingly homotopy equivalent, then $A\otimes\Zo$ is isomorphic to $B\otimes\Zo$.
\end{theorem}

\begin{proof}
	By assumption there exist $^*$-homomorphisms $\varphi: A \to B$ and $\psi: B \to A$ such that $\psi\circ \varphi$ and $\varphi \circ \psi$ are trace-preservingly homotopic to $\id_A$ and $\id_B$, respectively. 
	By Theorem \ref{maintheorem} it follows that
	\begin{align}
		(\psi \circ \varphi) \otimes \id_{\Zo} \approx_{\mathrm{u}} \id_{A \otimes \Zo} \quad \text{and} \quad (\varphi \circ \psi) \otimes \id_B \approx_{\mathrm{u}} \id_{B \otimes \Zo}. \notag
	\end{align}
	 A standard application of Elliott's intertwining argument yields that $\varphi\otimes\id_\Zo$ is approximately unitarily equivalent to an isomorphism.
	 In particular we obtain $A \otimes \Zo \cong B \otimes \Zo$.
\end{proof}

\begin{remark}
We take a moment to compare Theorem~\ref{theo:class} with the main result of \cite{Schafhauser2024}.
Under the assumption that one has a unital embedding $A\to B$ between separable unital nuclear $\Z$-stable \C-algebras that induces both a $KK$-equivalence and a bijection between the tracial state spaces, it is shown there that $A$ and $B$ are isomorphic.
It is not hard to see (regardless whether $A$ and $B$ are unital or not) that any embedding $\varphi: A\to B$ fitting into a trace-preserving homotopy equivalence must induce both a $KK$-equivalence and a bijection between $T^+(A)$ and $T^+(B)$.
Thus one can interpret our main result as a stably projectionless analog of Schafhauser's theorem, although admittedly a weaker one because the conclusion in Schafhauser's result is reached by merely assuming the existence of a nice embedding in one direction.
Upon assuming a trace-preserving homotopy equivalence, however, one immediately assumes the existence of nice embeddings in both directions, making our assumption conceptually much stronger in comparison.
It is an interesting question if Schafhauser's theorem has a more direct analog for non-unital \C-algebras.
\end{remark}

\newcommand{\etalchar}[1]{$^{#1}$}
\providecommand{\url}[1]{\texttt{#1}}
\providecommand{\urlprefix}{URL }

\end{document}